\documentclass[runningheads,a4paper]{llncs}
\usepackage{amsmath,amssymb}
\usepackage{nicefrac}
\usepackage{picture}
\usepackage{xcolor}
\usepackage{graphicx}
\usepackage{subfigure}
\usepackage{sidecap}
\usepackage[ruled,linesnumbered]{algorithm2e}
\usepackage{floatrow}
\usepackage{tikz}
\usetikzlibrary{fit,backgrounds,calc,positioning}
\usepackage[normalem]{ulem}

\usepackage{tabularx}
\newcolumntype{L}[1]{>{\raggedright\arraybackslash}p{#1}} 
\newcolumntype{C}[1]{>{\centering\arraybackslash}p{#1}} 
\newcolumntype{R}[1]{>{\raggedleft\arraybackslash}p{#1}} 
\usepackage{multirow}

\usepackage{standalone}
\standaloneconfig{
  mode=buildnew,
  build={latexoptions={-interaction=batchmode 
      -shell-escape -jobname \quote \buildjobname \quote}}
}

\DeclareMathOperator{\SX}{\mathcal{X}}
\DeclareMathOperator{\SSum}{\mathcal{S}}
\DeclareMathOperator{\R}{\mathbb{R}}
\DeclareMathOperator{\N}{\mathbb{N}}

\DeclareMathOperator{\FG}{\mathsf{G}}
\DeclareMathOperator{\FV}{\mathsf{V}}
\DeclareMathOperator{\FE}{\mathsf{E}}

\providecommand{\abs}[1]{\lvert#1\rvert}

\DeclareMathOperator*{\argmin}{arg\,min}

\newcommand{\pluseq}{\mathrel{{+}{=}}}
\newcommand{\minuseq}{\mathrel{{-}{=}}}

\newcommand\blfootnote[1]{%
  \begingroup
  \renewcommand\thefootnote{}\footnote{#1}%
  \addtocounter{footnote}{-1}%
  \endgroup
} 

\begin{document}
  
  %
  %
  \def\SSVM17SubNumber{1}
  %
  %
  \title{A Novel Convex Relaxation for Non-Binary Discrete Tomography}
  
  %
  %
  %
  \author{Jan Kuske$^\dagger$, Paul Swoboda$^\ddagger$ and Stefania Petra$^\dagger$}
  \institute{$^\dagger$MIG, Inst. Appl. Mathematics, Heidelberg University \\
  $^\ddagger$ Institute of Science and Technology (IST), Austria }
  
  \maketitle

  \begin{abstract}
    We present a novel convex relaxation and a corresponding inference algorithm for the non-binary discrete tomography problem, that is, reconstructing discrete-valued images from few linear measurements. 
    In contrast to state of the art approaches that split the problem into a continuous reconstruction problem for the linear measurement constraints and a discrete labeling problem to enforce discrete-valued reconstructions, we propose a joint formulation that addresses both problems simultaneously, resulting in a tighter convex relaxation.
    For this purpose a constrained graphical model is set up and evaluated using a novel relaxation optimized by dual decomposition.
    We evaluate our approach experimentally and show superior solutions both mathematically (tighter relaxation) and experimentally in comparison to previously proposed relaxations.
  \end{abstract}

\blfootnote{
  \noindent \textbf{Acknowledgments:} We gratefully acknowledge support by the DFG (German Science Foundation), Grant GRK 1653. This work is partially funded by the
  European Research Council under the European Unions 7th Framework Programme (FP7/2007-2013)/ERC grant
  agreement no 616160. The authors would like to thank Vladimir Kolmogorov
  for helpful discussions.
} 

\section{Introduction}

We study the discrete tomography problem, that is reconstructing a discrete-valued image from a small number of linear measurements (tomographic projections), see Figure~\ref{fig:discrete-tomography} for an illustration.
The main difficulty in reconstructing the original image is that there are usually far too few measurements, making the problem ill-posed.
Hence, it is common to search for a discrete-valued image that (i)~satisfies the measurements and (ii)~minimizes an appropriate energy function.

\begin{figure}[t]
        \floatbox[{\capbeside\thisfloatsetup{capbesideposition={right,center},capbesidewidth=0.55\textwidth}}]{figure}[\FBwidth]
        {\caption{Examplary discrete tomography example on three-valued image (white=$0$, gray=$1$, black=$2$) with three projection directions: horizontal, vertical and diagonal (right upper to lower left).
Horizontal sums are written on the right, vertical sums on the top and diagonal sums on the left and bottom.
        }}
        {\includegraphics[width=0.35\textwidth]{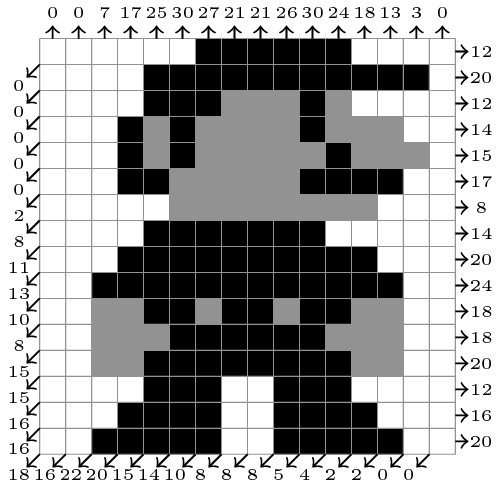}}
   \label{fig:discrete-tomography}
\end{figure}
More generally, the discrete tomography problem can be regarded as reconstructing a \emph{discrete-valued synthesis/analysis-sparse} signal from few measurements which is observed by \emph{deterministic} sensors $A$.
This is, in turn, a special instance of the compressed sensing problem~\cite{Foucart.2013}, for which it has been shown that discreteness constraints on the possible values of the reconstructed function can significantly reduce the number of required measurements~\cite{Keiper2016}.

However, the discreteness constraint leads to great computational challenges.
Simple outer convex relaxations coming from continuous scenarios are doomed to fail, as they will not output discrete solutions, unless the signal sparsity satisfies favourable relations and $A$ is well-conditioned on the class of sparse signals, e.g. a random matrix.
In fact, in most practical scenarios the projection matrices fall short of assumptions that underlie rigorous compressed sensing theory (like e.g.~the restricted isometry property \cite{Foucart.2013}), 
and standard algorithms from the continuous $\ell_{1}$-setting cannot be applied any more. 
Rounding continuous solutions will on the other hand render the solutions infeasible for the measurement constraints.

Therefore, algorithms exploiting the \textit{combinatorial structure} of the discrete tomography problem are necessary to successfully exploit discreteness as prior knowledge and to reduce the number of required measurements.

\paragraph{Related work.}
Several algorithms have been proposed to solve the discrete tomography problem.
Among them are
(i)~linear programming-based algorithms~\cite{TomoGC,LinearProgrammingBinaryTomographyWeber},
(ii)~belief propagation~\cite{DiscreteTomographyBP},
(iii)~network flow techniques~\cite{NetworkFlowBinaryTomographyBatenburg},
(iv)~convex-convave programming~\cite{DiscreteTomographyZisler,DiscreteTomographyConvexConcave}
(v)~evolutionary algorithms~\cite{batenburg2005evolutionary}
and other heuristic algorithms~\cite{DART,BinaryTomographyTriplaneCardiography,ParameterEstimationGibbsBinaryTomography,GaussMarkovPottsBinaryTomography}.
Not all approaches are applicable to the general discrete tomography problem we treat here:
algorithms~\cite{TomoGC,LinearProgrammingBinaryTomographyWeber,DiscreteTomographyBP,BinaryTomographyTriplaneCardiography,ParameterEstimationGibbsBinaryTomography,GaussMarkovPottsBinaryTomography} only support binary labels, while~\cite{DART,NetworkFlowBinaryTomographyBatenburg} solves only the feasiblity problem and does not permit any energy.
Algorithms~\cite{DiscreteTomographyZisler,DiscreteTomographyConvexConcave} are applicable to the setting we propose but are purely primal algorithms and do not output dual lower bounds (which we do).
Hence, one cannot judge proximity of solutions computed by~\cite{DiscreteTomographyZisler,DiscreteTomographyConvexConcave} to global optimal ones, prohibiting its use in branch and bound.
In case convex relaxations were considered~\cite{TomoGC,DiscreteTomographyZisler,LinearProgrammingBinaryTomographyWeber,NetworkFlowBinaryTomographyBatenburg}, they were less tight than the one we propose, leading to inferior dual bounds.

\paragraph{Contribution.}
We propose the (to our knowledge) first LP-based algorithm for the non-binary discrete tomography problem.
In particular, we
\begin{itemize}
  \item recast the discrete tomography problem as a Maximum-A-Posteriori inference problem in a graphical model with additional linear constraints coming from the tomographic projections in Section~\ref{sec:model},
  \item construct higher order factors in the graphical model such that a feasible solution to the higher order factors coincides with solutions feasible for the tomographic projections and
  \item present an efficient exact algorithm for solving the special case when exactly one ray constraint is present and the energy factorizes as a chain
    (such a problem will be called a one-dimensional discrete tomography problem) in Section~\ref{sec:one-dim-dt},
  \item decompose the whole problem into such subproblems and solve this decomposition with bundle methods in Section~\ref{sec:Algorithm}.
\end{itemize}
Our approach leads to significantly tighter bounds as compared to generalising previously proposed relaxations to the non-binary case, see Proposition~\ref{lemma:local-polytope-not-tight} in Section~\ref{sec:Algorithm} and experiments in Section~\ref{sec:Experiments}.
Code and datasets are available on GitHub\footnote{\url{https://github.com/pawelswoboda/LP_MP}}.

\paragraph{Notation.}
Let $[a,b] = \{a,\ldots,b\}$ be the set of natural numbers between $a$ and $b$.
For $x \in \R$ we denote by $\lfloor x \rfloor$ and $\lceil x \rceil$ the floor and ceiling function.

\section{Problem Statement}
\label{sec:model}
The discrete tomography problem we study consists in finding a discrete labeling $x \in \{0,1,\ldots,k-1\}^n$ such that
(i)~tomographic projection constraints given by $Ax = b$ with $A \in \{0,1\}^{m \times n}, b \in \N^m$ are fulfilled and
(ii)~$x$ minimizes some \emph{energy} $E : \{0,1,\ldots,k-1\}^{n} \rightarrow \R$.
We assume that $E$ factorizes according to a pairwise graphical model:
given a graph $\FG = (\FV,\FE)$, together with a \emph{label space} $\SX_{\FV} := \prod_{u \in \FV} \SX_u$, $\SX_u := \{0,1\ldots,k-1\}$ $\forall u \in \FV$, the energy is a sum of \emph{unary potentials} $\theta_u: \SX_u \rightarrow \R$ $\forall u \in \FV$ and \emph{pairwise} ones $\theta_{uv} : \SX_u \times \SX_v \rightarrow \R$ $\forall uv \in \FE$.
The full problem hence reads
\begin{equation}
  \label{eq:DiscreteTomography}
  \min_{x \in \SX_{\FV}} E(x) := \sum_{u \in \FV} \theta_{u}(x_u) + \sum_{uv \in \FE} \theta_{uv}(x_u,x_v) \quad \text{s.t.} \quad Ax = b\,.
\end{equation}

For the discrete tomography problem we usually choose $\FG$ to be a grid graph corresponding to the pixels of the image to be reconstructed, zero unary potentials $\theta_u \equiv 0$ $\forall u \in \FV$, \textbf{as no local information about the image values is known}, and pairwise potentials $\theta_{uv} = g(x_u - x_v)$ penalize intensity transitions, e.g. $g(\cdot) = \abs{\cdot}$ (TV) or $g(\cdot) = \min(1, \abs{\cdot})$ (Potts).
Such choice of pairwise potentials assigns small energy to labelings $x$ with a regular spatial structure.

\section{One-Dimensional Non-Binary Discrete Tomography}
\label{sec:one-dim-dt}
A natural decomposition of the discrete tomography problem~\eqref{eq:DiscreteTomography} consists of (i)~considering a subproblem for each ray constraint separately and (ii)~joining them together via Lagrangian variables.
We will study the first aspect below and the second one in the next section.
In particular, let $U = \{u_1,\ldots,u_n\} \subseteq \FV$ be the variables from a single ray constraint $x_{u_1} + \ldots + x_{u_n} = b$ corresponding to a row of the projection matrix $A$ in~\eqref{eq:DiscreteTomography}. 
Assume that pairwise potentials form a chain, i.e.\ they are $\theta_{u_i u_{i+1}}$, $i=1,\ldots,n-1$.
The \emph{one-dimensional discrete tomography problem} is
\begin{equation}
   \label{eq:ProjectionRay}
\begin{split}
    \min_{(x_1,\ldots,x_n) \in \SX_U} &  \sum_{u \in U} \theta_u(x_u) + \sum_{i \in [n-1]} \theta_{u_i,u_{i+1}}(x_{u_i},x_{u_{i+1}})  \\
    \text{s.t. }  & \sum_{u \in U} x_{u} = b\,. 
\end{split}
\end{equation}
We present an exact linear programming relaxation and an efficient message-passing routine to solve~\eqref{eq:ProjectionRay} below.

\subsection{Linear Programming Model}
The one-dimensional discrete tomography subproblem~\eqref{eq:ProjectionRay} could naively be solved by dynamic programming by going over all variables $u_1,\ldots,u_n$ sequentially.
This however would entail quadratic space complexity in the number of nodes in $U$, as the state space for variable $u_i$ would need to include costs for all possible labels $x_{u_{i}}$ and all values of the intermediate sum $\sum_{j=1}^{i-1} x_{u_j}$.
The latter sum can have $1+\sum_{j=1}^{i-1} (\abs{\SX}-1)$ possible values.
To achieve a better space complexity, we will recursively (i)~equipartition variables $u_1,\ldots,u_n$, (ii)~define LP-subproblems in terms of so-called \emph{counting factors} which are exact on each subpartition and (iii)~join them together to eventually obtain an exact LP-relaxation for~\eqref{eq:ProjectionRay}.
Our approach is inspired by~\cite{RecursiveCardinalityModels}.

\paragraph{Partition of variables.}
Given  the nodes $u_1,\ldots,u_n$, we choose an equipartition $\Pi_1 = \{u_1,\ldots,u_{\lfloor\nicefrac{n}{2}}\rfloor\}$ and $\Pi_2 = \{u_{\lfloor \nicefrac{n}{2}\rfloor+1} ,\ldots,u_n\}$.
We recursively equipartition $\Pi_1$ into $\Pi_{1,1}$ and $\Pi_{1,2}$ and do likewise for $\Pi_2$.
For $u_1,\ldots,u_8$ we obtain a recursive partitioning as in Figure~\ref{fig:SubSumIllustration}.

\begin{figure}
  \centering
  \includegraphics[width=0.45\linewidth]{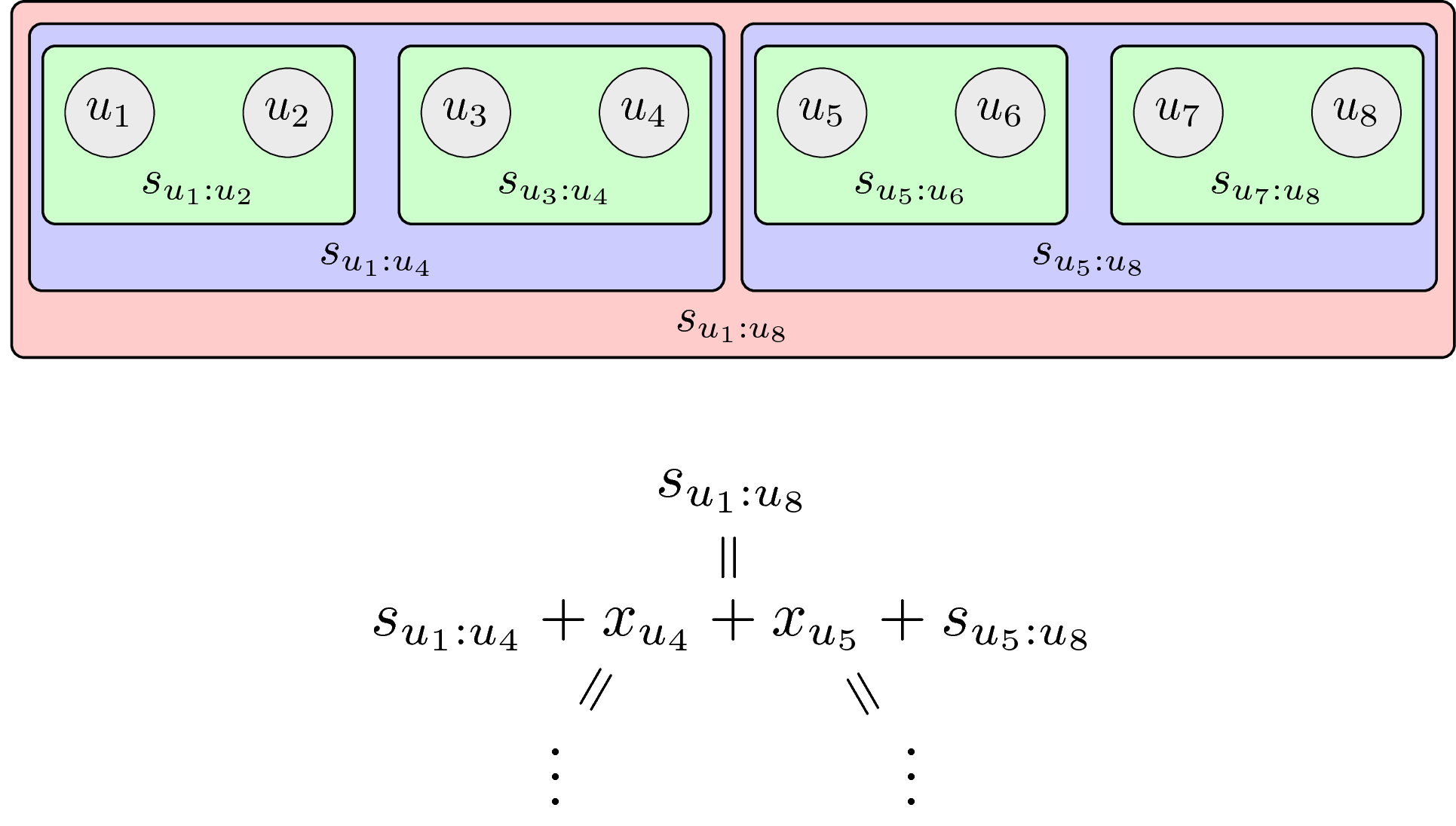}
  \caption[caption]{ 
  \emph{Partition of variables}:
  The chain of variables $U=\{u_1,\ldots,u_8\}$ (\textcolor{red}{red boxes}) is first partitioned into $\Pi_1 = \{u_1,u_2,u_3,u_4\}, \Pi_2 = \{u_5,u_6,u_7,u_8\}$ (\textcolor{blue}{blue boxes}), and then into $\Pi_{1,1} = \{u_1,u_2\}, \Pi_{1,2} = \{u_3,u_4\}, \Pi_{2,1} = \{u_5,u_6\}, \Pi_{2,2} = \{u_7,u_8\}$ (\textcolor{green}{green boxes}). \\\hspace{\textwidth}
  \emph{Counting sums:}
  The topmost sum $s_{u_1:u_8} = x_{u_2} + \ldots + x_{u_7}$ is composed of sums $s_{1:4}$ for partition $\Pi_1$ and $s_{5:8}$ for $\Pi_2$. Subsums are generated recursively for $s_{1:4}$ and $_{5:8}$ again.
  }
\label{fig:SubSumIllustration}
\end{figure}

\paragraph{Counting factors.}
Given an interval, a \emph{counting factor} holds the states of its left and right end and the value of the intermediate sum.
\begin{definition}[Counting label space]
  The counting label space for interval $[i,j]$ is $\SX_{i:j} := \SX_{u_i} \times \SSum_{i:j} \times \SX_{u_j}$ with $\SSum_{i:j} = \{0,1,\ldots,1+\sum_{l=i+1}^{j-1} (\abs{\SX_{u_l}}-1) \}$ holding all possible intermediate sums.
  A \emph{counting label} $x_{i:j}$ consists of the three components $(x_{u_i}, s_{i:j} , x_{u_j})$: its left endpoint label $x_{u_i}$, intermediate sum $s_{i:j} := x_{u_{i+1}} + \ldots + x_{u_{j-1}}$ and right endpoint label $x_{u_j}$.
\end{definition}
See again Figure~\ref{fig:SubSumIllustration} for the exemplary case $U = \{u_1,\ldots,u_8\}$. 

For interval $[i,j]$ there are $\abs{\SX_{i:j}} = \abs{\SX_{u_i}} \cdot \abs{\SX_{u_j}} \cdot \abs{\SSum_{i:j}}$ distinct counting labels.
We associate to each counting factor \emph{counting marginals} $\mu_{i:j}$ satisfying $\{\mu_{i:j} \in \R_+^{\abs{\SX_{i:j}}} : \sum_{x_{i:j} \in \SX_{i:j}} \mu_{i:j}(x_{i:j}) = 1\}$.

Assuming an uniform label space $\abs{\SX_u} = k$ $\forall u \in \FV$, the total space complexity of all counting factors is $O(k^3 \cdot n \cdot \log(n))$, hence subquadratic in the number of nodes in $U$.

\paragraph{Joining counting factors.}
Assume the partitioning of variables has produced two adjacent subsets $\Pi = \{u_i,\ldots,u_j \}$ and $\Pi' = \{u_{j+1},\ldots,u_{l}\}$, which were constructed from their common subset $\Pi \cup \Pi' \subseteq U$. 
The associated three counting factors with marginals $\mu_{i:j}, \mu_{j+1:l}$ and $\mu_{i:l}$ introduced above shall be consistent with respect to each other.

\begin{definition}[Label consistency]
  Label $x_{i:l} \in \SX_{i:l}$, $x_{i:j} \in \SX_{i:j}$ and $x_{j+1:l}$ are \emph{consistent} with each other, denoted by $x_{i:j},x_{j+1:l} \sim x_{i:l}$ iff 
(i)~left endpoint labels of $x_{i:j}$ and $x_{i:l}$ match, (ii)~right endpoint labels of $x_{j+1:l}$ and $x_{i:l}$ match and
(ii)~intermediate sums match $s_{i:l} = s_{i:j} + x_{u_{j}} + x_{u_{j+1}} + s_{j+1:l}$.
\end{definition}
We enforce this by introducing a \emph{higher order marginal} $\mu_{i:j:l} \in \R_{+}^{\SX_{i:j} \times \SX_{j+1:l}}$ to bind together $\mu_{i:j}, \mu_{j+1:l}$ and $\mu_{i:l}$.
\begin{align}
  \label{eq:CountingFactorConstraintsLeft}
    \sum_{x_{j+1:l}} \mu_{i:j:l}(x_{i:j},x_{j+1:l}) & = \mu_{i:j}(x_{i:j}) & \forall x_{i:j} \in \SX_{i:j} \\
  \label{eq:CountingFactorConstraintsRight}
    \sum_{x_{i:j}} \mu_{i:j:l}(x_{i:j},x_{j+1:l}) & = \mu_{j+1:l}(x_{j+1:l}) & \forall x_{j+1:l} \in \SX_{j+1:l} \\
  \label{eq:CountingFactorConstraintsUp}
    \sum_{x_{i:j},x_{j+1:l} \sim x_{i:l} }\mu_{i:j:l}(x_{i:j},x_{j+1:l}) & = \mu_{i:l}(x_{i:l}) & \forall x_{i:l} \in \SX_{i:l}
\end{align}
The recursive arrangement of counting factors is illustrated in Figure~\ref{fig:TreeIllustration}.
\begin{remark}
The constraints between $\mu_{i:j:l}$ and $\mu_{i:j}$ and $\mu_{j+1:l}$ are analoguous to the marginalization constraints between pairwise and unary marginals in the local polytope relaxation for pairwise graphical models~\cite{Werner.2007}.
The constraints between $\mu_{i:j:l}$ and $\mu_{i:l}$ however are different.
  Hence, specialized efficient solvers for inference in graphical models cannot be applied.
\end{remark}

\begin{figure}[t]
  \centering
  \includegraphics[width=0.5\linewidth]{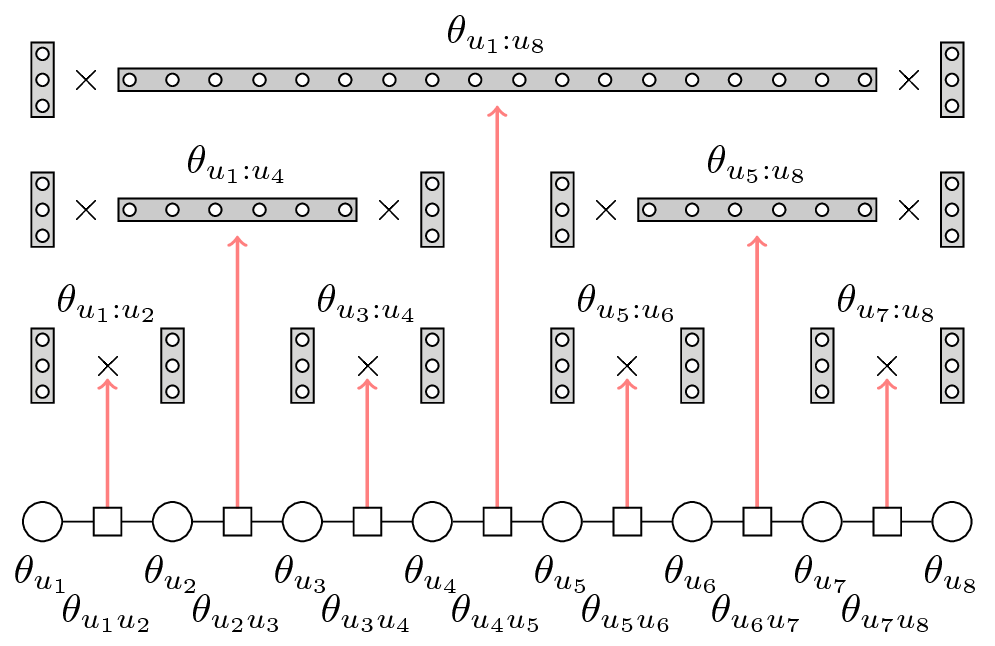}
  \caption{ \label{fig:TreeIllustration}
    Illustration of the tree construction used for solving the one-dimensional tomography problem~\eqref{eq:ProjectionRay} with $\ell = 3$, i.e. three labels.
    On the bottom is the original chain, above it are three layers of counting factors.
    Each counting factor 
    keeps track of the label $x_u$ (left vertical bar), $x_v$ (right vertical bar) and the intermediate sum $\sum_{w \in U: u < w < v} x_w$ (middle horizontal bar).
  }
\end{figure}
\paragraph{Costs.}
Above we have described the polytope for the one-dimensional discrete tomography problem~\eqref{eq:ProjectionRay}.
The LP-objective consists of vectors $\theta_{i:j}$ for each counting marginal and $\theta_{i:j:l}$ for each higher order marginal.
Accounting for the pairwise costs in~\eqref{eq:ProjectionRay} we set $\theta_{i:j}(x_{i:j}) := \begin{cases} \theta_{u_i u_j}(x_{u_i},x_{u_j}), & i+1=j \\ 0,& \text{otherwise} \end{cases}$ for the counting factors and
for the higher order factors we set $\theta_{i:j:l}(x_{i:j},x_{j+1:l}) := \theta_{u_{j}u_{j+1}}(x_{u_j},x_{u_{j+1}})$.
For the projection constraint in~\eqref{eq:ProjectionRay} we set costs of the top counting marginal as $\theta_{1:n}(x_{1:n}) := \begin{cases} 0, & x_{u_1} + s_{1:n} + x_{u_n} = b \\ \infty,& \text{otherwise} \end{cases}$.
\subsection{Message Passing Algorithm}
Above we have introduced a linear program formulation for the one-dimensional discrete tomography problem~\eqref{eq:ProjectionRay}.
While it is possible to solve it with a standard LP-solver, doing so would be slow.
As the counting factors and the higher order marginals connecting them form a tree, it is possible to devise a message passing algorithm optimizing~\eqref{eq:ProjectionRay} exactly.
First, this implies that the linear programming relaxation for~\eqref{eq:ProjectionRay} is exact, as message passing amounts to optimizing the Lagrangian dual of this same relaxation.
Second, marginals do not need to be held explicitly, holding messages is enough. The size of all messages equals the size of all counting factors, hence giving again subquadratic space complexity.

Message passing for~\eqref{eq:ProjectionRay} is detailed in Algorithm~\ref{alg:message-passing-1d}.
It proceeds by first computing \emph{up messages} from adjacent fine subsets to coarser subsets (i.e.\ going up the tree in Figure~\ref{fig:TreeIllustration}) and afterwards computing \emph{down messages} from coarse subsets to their equipartition (i.e.\ going down the tree in Figure~\ref{fig:TreeIllustration}).
Messages reparametrize costs of counting and higher order factors.

\paragraph{Reparametrization}
Let indices $i<j<l$ be given, where $[i:j]$, $[j+1:l]$ and $[i:l]$ are subsets generated by the recursive partitioning.
Let messages $\phi_{i:j:l}^{\leftarrow}, \phi_{i:j:l}^{\rightarrow}, \phi_{i:j:l}^{\uparrow}$ correspond to constraints~\eqref{eq:CountingFactorConstraintsLeft},~\eqref{eq:CountingFactorConstraintsRight} and~\eqref{eq:CountingFactorConstraintsUp} respectively.
Messages $\phi$ act on (reparametrize) costs $\theta$ as
\begin{align}
  \theta^{\phi}_{i:j:l}(x_{i:j},x_{j+1:l}) &\pluseq \phi^{\leftarrow}_{i:j:l}(x_{i:j}) & \theta^{\phi}_{i:j}(x_{i:j}) &\minuseq  \phi^{\leftarrow}_{i:j:l}(x_{i:j}) \\
  \theta^{\phi}_{i:j:l}(x_{i:j},x_{j+1:l}) &\pluseq \phi^{\rightarrow}_{i:j:l}(x_{j+1:l}) & \theta^{\phi}_{j+1:l}(x_{j+1:l}) &\minuseq  \phi^{\rightarrow}_{i:j:l}(x_{j+1:l}) \\
  \theta^{\phi}_{i:j:l}(x_{i:j},x_{j+1:l}) &\pluseq \phi^{\uparrow}_{i:l}(x_{i:l}) & \theta^{\phi}_{i:l}(x_{i:l}) &\minuseq  \phi^{\uparrow}_{i:j:l}(x_{i:l})
\end{align}
\begin{algorithm}
  \caption{Message passing for one-dimensional discrete tomography}
  \label{alg:message-passing-1d}
  \textbf{Up messages:} \\ 
  \For{$[i,j] \cup [j+1,l] = [i,l]\in \Pi$ \text{in ascending order}} {
    $\begin{matrix}
      \phi^{\leftarrow}_{i:j:l} &=& \theta_{i:j}\\ \phi^{\rightarrow}_{i:j:l} &=& \theta_{j+1:l} \\
    \end{matrix}$
  \tcp{Send messages to higher order counting factor $\theta_{i:j:l}^{\phi}$}
  \label{eq:MessageUpFromHigherOrder}
  $ \phi^{\uparrow}_{i:j:l}(x_{i:l}) = \hspace{-0.2cm} \min\limits_{x_{i:j}, x_{j+1:l} \sim x_{i:l}} \theta_{i:j:l}^{\phi}(x_{i:j},x_{j+1:l}) $
  \parbox{4.3cm}{
    \tcp{Send~message~from~higher}
    \tcp{order~to~counting~factor~$\theta_{i:l}^{\phi}$}}
  } 
  $x_{1:n}^* \in \argmin_{x_{1:n} \in \SX_{1:n}} \theta_{1:n}(x_{1:n})$ 
  \tcp*[r]{optimum of top counting factor} 
  \textbf{Down messages:}  \\
  \For{$[i,j] \cup [j+1,l] = [i,l]\in \Pi$ in descending order} {
    $\phi_{i:j:l}^{\uparrow} = \theta_{i:j:l} $ \tcp*[r]{Send message to higher order factor}
  $  x_{i:j}^*,x_{j+1:l}^* \in \argmin\limits_{ x_{i:j},x_{j+1:l}  \sim x_{i:l}^* } \theta_{i:j:l}^{\phi}(x_{i:j},x_{j+1:l}) $ \tcp*[r]{compute optimal labels} 
  }
\end{algorithm}
\paragraph{Fast message computation}
Naively computing one up messages would result in time complexity $O(\ell^5 \cdot n^2)$, which would make the algorithm prohibitively slow.
We will describe a fast message computation technique for~\eqref{eq:MessageUpFromHigherOrder}, which uses the structure of the corresponding linear constraints~\eqref{eq:CountingFactorConstraintsUp} and relies on the latent factorization of $\theta_{i:j:l}^{\phi}$.
Specifically, when we fix the endpoints $x_{u_i}, x_{u_j}$ of interval $[i,j]$ and $x_{u_{j+1}}, x_{u_l}$ of $[j+1,l]$,~\eqref{eq:MessageUpFromHigherOrder} becomes
\begin{equation}
  \label{eq:MinSumConvolution}
  \phi^{\uparrow}_{i:j:l}(x_{i:l}) = \min_{s_{i:j} + s_{j+1:l} = s_{i:l}-x_{u_{j}} - x_{u_j} } \theta_{u_j,u_{j+1}}(x_{u_j},x_{u_{j+1}}) + \phi^{\leftarrow}_{i:j}(x_{i:j}) + \phi^{\rightarrow}_{j+1:l}(x_{j+1:l})\,,
\end{equation}
Problem~\eqref{eq:MinSumConvolution} is an instance of the min-sum convolution problem: Given $a,b \in \R^n$, compute $c \in \R^{2n-1}$, where $c_{i} = \min_{j \leq i} (a_j + b_{i-j})$.
This can be seen by replacing $\phi^{\leftarrow}$ by $a$, $\phi^{\rightarrow}$ by $b$ and noting that $\theta_{u_j, u_{j+1}}$ is a constant, as $x_{u_j}$ and $x_{u_{j+1}}$ were fixed.
For the min-sum convolution problem efficient algorithms~\cite{BussieckMaximalConvolution} were proposed with expected running time $O(n \log(n))$ under the assumption that sorting $a$ and $b$ results in permutations occurring with uniform probability.
Problem~\eqref{eq:MinSumConvolution} can be efficiently computed by performing $O(\ell^4)$ min-sum convolutions (one convolution for every choice of endpoints).
\begin{remark}[Comparison to~\cite{RecursiveCardinalityModels}]
  While our approach for solving~\eqref{eq:ProjectionRay} is inspired by~\cite{RecursiveCardinalityModels}, it is notably different:
  (i)~our model includes pairwise potentials forming a chain, while~\cite{RecursiveCardinalityModels} assumes that pairwise potentials do not occur between neighboring subsets. This necessitates to store left and right endpoints in counting factors.
  (ii)~\cite{RecursiveCardinalityModels} optimizes a different objective: they solve the sum-product version of~\eqref{eq:ProjectionRay} (i.e.\ they exchange min by $+$ and $+$ by $\cdot$ in~\eqref{eq:ProjectionRay}).
  This allows~\cite{RecursiveCardinalityModels} to use fast Fourier transforms for message computations, instead of the harder min-sum convolution problems.
\end{remark}

\section{Discrete Tomography Graphical Model}
\label{sec:Algorithm}
The discrete tomography problem~\eqref{eq:DiscreteTomography} consists of m = \#rows$(A)$ distinct one-dimensional subproblems~\eqref{eq:ProjectionRay}.
We connect all subproblems~\eqref{eq:ProjectionRay} via Lagrangian variables into one large problem.
This procedure is called dual decomposition, see~\cite{DDforInference} for an introduction.
Specifically, in our discrete tomograpy problems subproblems only share variables $v \in \FV$, but not edges $e \in \FE$ (shared edges can be handled analoguously).
Then for each node $u \in \FV$ which participates in the $i$-th subproblem, we introduce the Lagrangian variable $\lambda_{i,u} \in \R^{\abs{\SX_u}}$.
The $i$-th subproblem then consists of solving~\eqref{eq:ProjectionRay} with the subset of variables $U_i$, where the unary potentials are the Lagrangian variables $\theta_u = \lambda_{i,u}$.
We denote its energy by $E_{i}(\cdot | \lambda_i)$.
The overall problem is
\begin{equation}
  \tag{CTG}
  \label{eq:dt-decomposition}
  \max_{\lambda_1,\ldots,\lambda_m} \sum_{i=1}^m \min_{x \in \SX_{U_i}} E_i(x | \lambda_i) \quad \text{s.t.} \quad \sum_{i} \lambda_{i,u} \equiv 0 \quad \forall u \in \FV\,.
\end{equation}
An exemplary  $4\times 4$ model with eight subproblems coming from two projection directions can be seen in Figure~\ref{fig:JointGraphicalModel}.

\begin{SCfigure}
   \centering
    \includegraphics[width=0.37\linewidth]{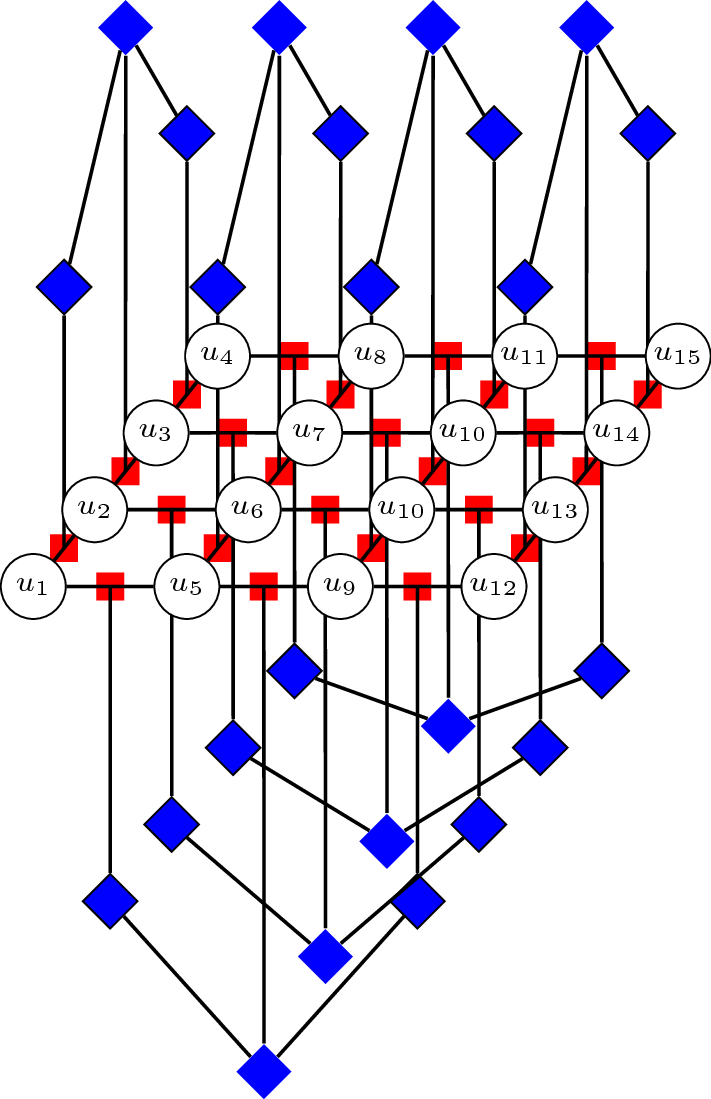}
      \caption{
      Illustration of a complete graphical model for the discrete tomography problem with projections along rows and columns of the underlying $4 \times 4$ grid.
      Graphical model consisting of submodel for energy $E$ in~\eqref{eq:DiscreteTomography}, indicated by unary (circles) and pairwise (red rectangles) potentials on the $4 \times 4$ grid,
      and counting factors (blue diamond) for the projection constraints $Ax = b$ using the tree construction above the grid (horizontal projection) and below (vertical projection).
      The higher order potentials are left out for the sake of clarity of presentation.
   }
   \label{fig:JointGraphicalModel}
\end{SCfigure}

\paragraph{Optimization of relaxation~\eqref{eq:dt-decomposition}.}
To maximize~\eqref{eq:dt-decomposition} we use the bundle solver ConicBundle\footnote{\url{https://www-user.tu-chemnitz.de/~helmberg/ConicBundle/}} to find optimal Lagrangian variables $\lambda$ and Algorithm~\ref{alg:message-passing-1d} to find solutions to the one-dimensional subproblems.
The bundle method will only give us a dual lower bound to the value of the optimal reconstruction.

\paragraph{Primal solution.}
To obtain a feasible reconstruction, we solve a reduced problem by excluding labels with high cost:
Given dual variables $\lambda_i$, let $x^*$ be the optimal solution to the $i$-th subproblem on variables $U_i \subseteq \FV$. For each label $x_u \in \SX_u, u\in U_i$, $x_u \neq x_u^*$ we compute the energy $x'^* \in \argmin_{\{x' \in \SX_{U} : x_u = x'_u\}} E_i(x' | \lambda_i)$ of the minimal reconstruction for subproblem $i$ when the label at $u$ is fixed to $x_u$ (this value can be read off from the reparametrization output by Algorithm~\ref{alg:message-passing-1d}).
Only if the gap $E_i(x'^* | \lambda_i) - E_i(x^* | \lambda_i)$ is smaller than some given threshold, we consider the label $x_u$.
We construct the discrete tomograpy problem on this reduced set of possible labelings and solve the problem with CPLEX~\cite{cplex}.

\paragraph{Comparison to previously used relaxation.}
It can be shown that the algorithms in~\cite{TomoGC,DiscreteTomographyZisler,LinearProgrammingBinaryTomographyWeber,NetworkFlowBinaryTomographyBatenburg} use the following relaxation.
\begin{equation}
  \tag{STD}
  \label{eq:local-polytope-dt}
  \begin{array}{rl}
    \min_{\mu \geq 0} & \sum_{u \in \FV} \langle \theta_u, \mu_u \rangle + \sum_{uv \in \FE} \langle \theta_{uv}, \mu_{uv} \rangle \\
    \text{s.t.} 
    & \sum_{x_u \in \SX_u} \mu_u(x_u) = 1 \quad \forall u \in \FV \\
    & \begin{array}{c} 
    \sum_{x_u \in \SX_u} \mu_{uv}(x_u,x_v) = \mu_{v}(x_v) \quad \forall x_v \in \SX_v \\ 
    \sum_{x_v\in \SX_v} \mu_{uv}(x_u,x_v) = \mu_{u}(x_u)  \quad \forall x_u \in \SX_u
    \end{array} \quad \forall uv \in \FE \\
      & \sum_{u \in \FV} A_{iu} \cdot \left(\sum_{x_u \in \SX_u} x_u \cdot \mu_u(x_u) \right) = b_i \quad i=1,\ldots,m \,.
  \end{array}
\end{equation}
This relaxation~\eqref{eq:local-polytope-dt} is the straigtforward generalization of the local polytope relaxation~\cite{Werner.2007} to the discrete tomography problem.
The only difference are the linear constraints in the last line of~\eqref{eq:local-polytope-dt}.
When specialized to the one-dimensional discrete tomography problem~\eqref{eq:ProjectionRay}, the difference between~\eqref{eq:local-polytope-dt} and our approach is:
for~\eqref{eq:local-polytope-dt} the tomographic projections are directly enforced through the unary marginals $\mu_u, u \in \FV$ instead of enforcing them through the counting factors and higher order ones as we did in Section~\ref{sec:one-dim-dt}.
This more simplistic relaxation~\eqref{eq:local-polytope-dt} is however less tight.

\begin{proposition}
  \label{lemma:local-polytope-not-tight}
  Relaxation~\eqref{eq:local-polytope-dt} is less tight than~\eqref{eq:dt-decomposition}.
\end{proposition}
\begin{proof}
  Relaxation~\eqref{eq:local-polytope-dt} is equivalent to applying it to each tomographic projection separately and then joining every subproblem by Lagrangian variables as we did with our approach above~\eqref{eq:dt-decomposition}, see~\cite[Section 1.6]{DDforInference}.
  Hence, it is enough to show that~\eqref{eq:local-polytope-dt} is not tight in the one-dimensional case~\eqref{eq:ProjectionRay}.
  We give a counter-example.
  Assume $\SX_u = \{0,1\}$ $\forall u \in U$ and we are given Potts pairwise potentials $\theta_{uv}(x_u,x_v) = \begin{cases} 0,& x_u = x_v \\ 1,& x_u \neq x_v \end{cases}$ and zero unary potentials $\theta_u \equiv 0$.
    Set unary marginals $\mu_u(1) = \frac{b}{\abs{U}}$ and $\mu_u(0) = 1 - \mu_u(1)$ $\forall u \in U$ and pairwise marginals as $\mu_{uv}(x_u,x_v) = \begin{cases} \mu_u(x_u),& x_u = x_v \\ 0,& x_u \neq x_v \end{cases}$.
      Such marginals are feasible to~\eqref{eq:local-polytope-dt}, yet give cost $0$. 
  On the other hand for e.g. $b=1$ and $\abs{U} > 1$ there must be at least one label transition, which the Potts potential penalizes with cost $1$.\qed
\end{proof}

\section{Experiments}
\label{sec:Experiments}
\label{sec:experiments}
\paragraph{Test images.}
We used $200$ randomly generated $32 \times 32$ images with three distinct intensity values $\{0,1,2\}$, examples of which can be seen in Figure~\ref{fig:RandomImages}.
Matrices $A$ for the tomographic projections were constructed as in~\cite{Roux.2014}.
For each test image we consider two tomographic problems: (i)~measuring along horizontal and vertical directions or (ii)~measuring along horizontal, vertical and two diagonal directions (left upper to right lower and left lower to right upper corner).
This gives 400 test problems in total.
Potentials for energy $E$ in~\eqref{eq:DiscreteTomography} are: unary potentials are zero, while pairwise ones are $\theta_{uv} = \abs{x_u - x_v}$ (that corresponds to TV).
Due to integrality of all costs, optimality is ascertained through a duality gap $<1$.
\newcommand{\TestSetImageWidth}{1.15cm}
\newcommand{\TestSetImageHspace}{-0.15cm}
\begin{figure}
  \includegraphics[width=\TestSetImageWidth]{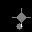}
  \hspace*{\TestSetImageHspace}
  \includegraphics[width=\TestSetImageWidth]{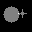}
  \hspace*{\TestSetImageHspace}
  \includegraphics[width=\TestSetImageWidth]{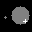}
  \hspace*{\TestSetImageHspace}
  \includegraphics[width=\TestSetImageWidth]{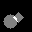}
  \hspace*{\TestSetImageHspace}
  \includegraphics[width=\TestSetImageWidth]{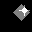}
  \hspace*{\TestSetImageHspace}
  \includegraphics[width=\TestSetImageWidth]{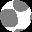}
  \hspace*{\TestSetImageHspace}
  \includegraphics[width=\TestSetImageWidth]{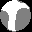}
  \hspace*{\TestSetImageHspace}
  \includegraphics[width=\TestSetImageWidth]{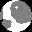}
  \hspace*{\TestSetImageHspace}
  \includegraphics[width=\TestSetImageWidth]{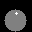}
  \hspace*{\TestSetImageHspace}
  \includegraphics[width=\TestSetImageWidth]{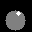}
  \caption{\label{fig:RandomImages} These images are examples from our testset of $200$ images with size $32\times 32$. Black represents value $0$, gray $1$ and white $2$. }
\end{figure}
\paragraph{Algorithms.}
We identify our solvers by a prefix \textbf{\{CTG\textbar STD\}} depending on whether~\eqref{eq:dt-decomposition} or~\eqref{eq:local-polytope-dt} is solved and by a suffix
\textbf{\{CB\textbar relax\textbar BB\}} depending on whether ConicBundle, CPLEX~\cite{cplex} or CPLEX with branch and bound enabled was utilized.
This gives in total 5 solvers: 
\textbf{CTG\_CB}, \textbf{CTG\_relax}, \textbf{CTG\_BB}, \textbf{STD\_relax} and \textbf{STD\_BB}.
We set a timelimit of 1 hour for all algorithms.

Unfortunately, CPLEX cannot solve problems larger than $32 \times 32$.
When solving the relaxation~\eqref{eq:dt-decomposition}, it already consumes multiple GB of memory for $32 \times 32$ images.
Solving~\eqref{eq:local-polytope-dt} on the other hand leads to low memory consumption, but CPLEX takes too much time for larger problems ($>1$ hour).
Hence, to have a baseline, we stick to $32 \times 32$ images.

\paragraph{Results.}
We have proved in Proposition~\ref{lemma:local-polytope-not-tight} that relaxation~\eqref{eq:local-polytope-dt} is less tight than our relaxation~\eqref{eq:dt-decomposition}. 
In fact, the first line in Table~\ref{table:comparison} shows that this occurs $350$ times.
Furthermore, our tighter relaxation also actually helps in giving optimality certificates.
In Table~\ref{table:optimality} we confirm this numerically: \textbf{STD\_relax} can provide optimality certificates $53$ times, while \textbf{CBC\_CB} and \textbf{CTG\_relax} can do so in total $205$ times.
Interestingly, when using the branch and bound capabilities of CPLEX, the picture changes and \textbf{STD\_BB} outperforms \textbf{CTG\_BB}. This is probably due to the fact that CPLEX can solve the underlying relaxation~\eqref{eq:local-polytope-dt} much faster than~\eqref{eq:dt-decomposition}. We conjecture that the picture will change if the more efficient implementation \textbf{CBC\_CB} is used as a bounds provider inside a branch and bound solver. This is however outside the scope of our work.

In Figure~\ref{fig:result} we give a detailed plot on how much our relaxation~\eqref{eq:dt-decomposition} improved upon~\eqref{eq:local-polytope-dt}.

Also, our relaxation helps in reconstructing the signal. 
Out of 238 instances, where our heuristic could find an optimal integral solution (third line in Table~\ref{table:comparison}) there were 12 cases, where only our heuristic could do so (second line in Table~\ref{table:comparison}).

\setlength\belowcaptionskip{-0.9cm}
\begin{table}
\vspace{-0.5cm}
  \caption{Number of instances where duality gap $<1$ (optimality).}
  \label{table:optimality}
  \centering
  \begin{tabular}{L{0.24\textwidth}|C{0.15\textwidth}|C{0.15\textwidth}|C{0.15\textwidth}|C{0.15\textwidth}|C{0.15\textwidth}}
    & STD relax & STD BB & CTG CB & CTG relax & CTG BB \\
    \hline
    \multirow{2}{*}{duality gap "$<1$"} & \multirow{2}{*}{$53$} & \multirow{2}{*}{$243$} & $178$ & $154$ & \multirow{2}{*}{$182$} \\
    & & & \multicolumn{2}{c|}{$205$} 
  \end{tabular}
\vspace{-0.2cm}
\end{table}
\begin{table}
\vspace{-0.8cm}
  \caption{Comparison of bounds and primal solutions obtained by~\eqref{eq:local-polytope-dt} or~\eqref{eq:dt-decomposition}.}
  \label{table:comparison}
  \centering
  \begin{tabular}{L{0.83\textwidth}|C{0.2\textwidth}}
    & \#Instances \\
    \hline
    \eqref{eq:dt-decomposition} $>$ \eqref{eq:local-polytope-dt} (our relaxation yields \textbf{strictly better} lower bound) & $350$ \\
    \hline
    our heuristic \textbf{(only)} found \textbf{optimal} integral solution & $12$ \\
    \hline
    our heuristic found \textbf{optimal} integral solution & $238$
  \end{tabular}
\end{table}

\begin{figure}
  \centering
  \includegraphics[width=0.99\linewidth]{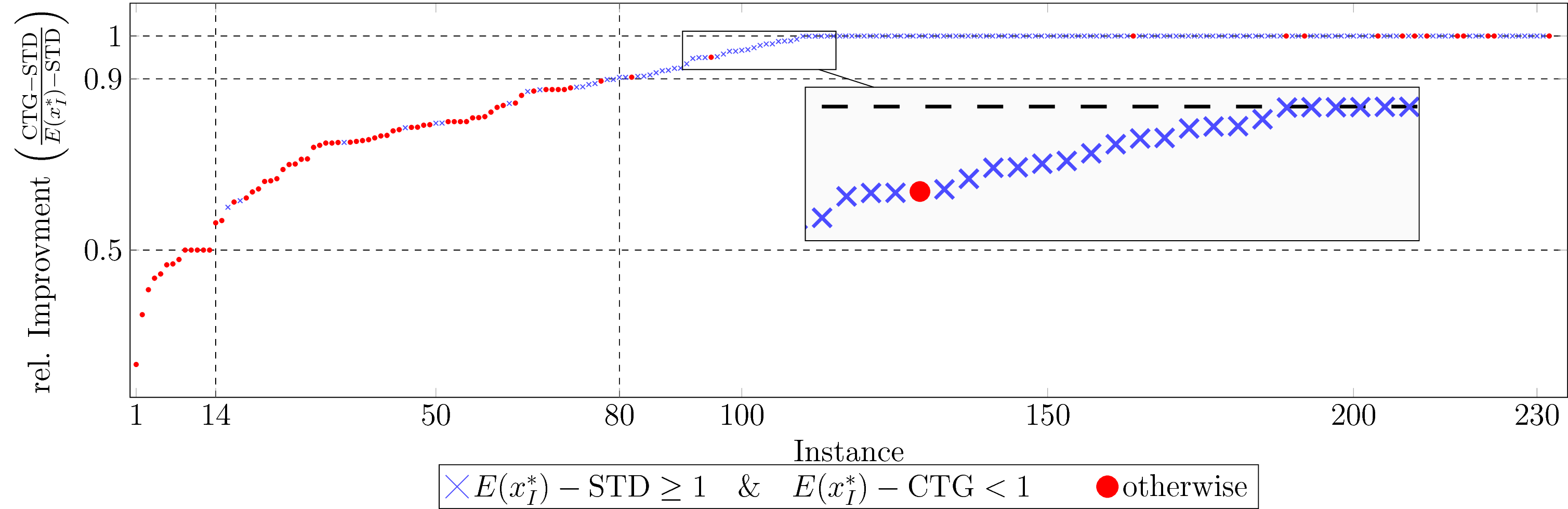}
  \caption{\label{fig:result}
  Comparison of lower bounds~\eqref{eq:local-polytope-dt} and~\eqref{eq:dt-decomposition}.
  We show relative improvement $\nicefrac{\eqref{eq:dt-decomposition} - \eqref{eq:local-polytope-dt}}{E(x^*) - \eqref{eq:local-polytope-dt}}$ on all problems, where we knew by either method the true optimal solution $E(x^*)$, but where~\eqref{eq:local-polytope-dt} is not tight ($\eqref{eq:local-polytope-dt} < E(x^*)$).
Lower bound~\eqref{eq:local-polytope-dt} was computed by \textbf{STD\_relax}, while~\eqref{eq:dt-decomposition} was computed by \textbf{CTG\_CB} and \textbf{CTG\_relax}.
  A marker close to zero means no improvement and close to one means our relaxation solved the instance exactly. 
  We marked points with a blue cross if only~\eqref{eq:dt-decomposition} but not~\eqref{eq:local-polytope-dt} achieved a duality gap $<1$, i.e. optimality, and used a red circle otherwise.
  For almost all instance we have an improvement of $0.5$ and for more than half of the instances an improvement of $0.9$.}
\end{figure}

\section{Conclusion}
We have proposed a novel convex relaxation and an accompanying algorithm for the non-binary discrete tomography problem.
We have showed theoretically and empirically that our novel relaxation is tighter than the traditionally used relaxation.
Solving our new relaxation helps in decoding tomographic reconstructions.
{\small
\bibliographystyle{abbrv}
\bibliography{literatur}
}

\end{document}